\documentclass[reqno, a4paper]{amsart}

\usepackage{fixltx2e}
\usepackage[T1]{fontenc}
\usepackage[british]{babel}
\usepackage{amsfonts, amssymb, amsthm, amsmath}
\usepackage[all]{xy}
\usepackage{hyperref}
\usepackage{mathbbol}
\usepackage{calrsfs}
\usepackage{mathabx}
\usepackage{wasysym}

\theoremstyle{plain}
\newtheorem{lemma}[subsection]{Lemma}
\newtheorem{theorem}[subsection]{Theorem}
\newtheorem{proposition}[subsection]{Proposition}
\newtheorem{corollary}[subsection]{Corollary}

\theoremstyle{definition}
\newtheorem{example}[subsection]{Example}
\newtheorem{examples}[subsection]{Examples}
\newtheorem{definition}[subsection]{Definition}

\newtheorem{remark}[subsection]{Remark}

\newcommand{\noproof}{\hfill\qed}
\newcommand{\defn}{\textbf}
\newcommand{\comp}{\raisebox{0.2mm}{\ensuremath{\scriptstyle{\circ}}}}
\newcommand{\join}{\vee}

\newcommand{\tensor}{\diamond}
\renewcommand{\implies}{~$\Rightarrow$~}

\newcommand{\mus}[2]{\left\langle\begin{smallmatrix} #1 \\ #2 \end{smallmatrix}\right\rangle}
\newcommand{\muss}[3]{\left\langle\begin{smallmatrix} {#1} \\ {#2} \\ {#3}\end{smallmatrix}\right\rangle}

\DeclareMathOperator{\Ker}{Ker}

\renewcommand{\S}{\ensuremath{\mathrm{S}}}

\newcommand{\CKL}{\ldbrack K,L\rdbrack}

\newcommand{\normal}{\ensuremath{\lhd}}
\newcommand{\protosplit}{\ensuremath{\LHD}}

\renewcommand{\k}{\ensuremath{\mathbf{k}}}
\renewcommand{\l}{\ensuremath{\mathbf{l}}}

\newcommand{\C}{\ensuremath{\mathcal{C}}}
\newcommand{\D}{\ensuremath{\mathcal{D}}}

\newcommand{\CC}{\ensuremath{\mathbb{C}}}
\newcommand{\DD}{\ensuremath{\mathbb{D}}}
\newcommand{\K}{\ensuremath{\mathbb{K}}}

\newcommand{\Gp}{\ensuremath{\mathsf{Gp}}}

\newcommand{\Leibniz}{\ensuremath{\mathsf{Leibniz}}}
\newcommand{\Lie}{\ensuremath{\mathsf{Lie}}}
\newcommand{\NARng}{\ensuremath{\mathsf{NARng}}}
\newcommand{\Pt}{\ensuremath{\mathsf{Pt}}}
\newcommand{\Poisson}{\ensuremath{\mathsf{Poisson}}}
\newcommand{\Rng}{\ensuremath{\mathsf{Rng}}}
\newcommand{\CRng}{\ensuremath{\mathsf{CRng}}}

\newcommand{\CI}[1]{{\rm (COI\,#1)}}

\newcommand{\NH}{{\rm (NH)}}

\newcommand{\SH}{{\rm (SH)}}

\def\pullback{
 \ar@{-}[]+R+<6pt,-1pt>;[]+RD+<6pt,-6pt>%
 \ar@{-}[]+D+<1pt,-6pt>;[]+RD+<6pt,-6pt>}
 
\def\pushout{%
 \ar@{-}[]+L+<-6pt,1pt>;[]+LU+<-6pt,6pt>%
 \ar@{-}[]+U+<-1pt,6pt>;[]+LU+<-6pt,6pt>}

\def\splitpullback{%
 \ar@{-}[]+R+<6pt,-.51ex>;[]+RD+<6pt,-6pt>%
 \ar@{-}[]+D+<.51ex,-6pt>;[]+RD+<6pt,-6pt>}

\def\skewpullback{%
 \ar@{-}[]+LD+<-6pt,-6pt>;[]+LDD+<-6pt,-15.5pt>%
 \ar@{-}[]+D+<-1pt,-6pt>;[]+LDD+<-6pt,-15.5pt>}

\newdir{>>}{{}*!/3.5pt/:(1,-.2)@^{>}*!/3.5pt/:(1,+.2)@_{>}*!/7pt/:(1,-.2)@^{>}*!/7pt/:(1,+.2)@_{>}}
\newdir{ >>}{{}*!/8pt/@{|}*!/3.5pt/:(1,-.2)@^{>}*!/3.5pt/:(1,+.2)@_{>}}
\newdir{ |>}{{}*!/-3.5pt/@{|}*!/-8pt/:(1,-.2)@^{>}*!/-8pt/:(1,+.2)@_{>}}
\newdir{ >}{{}*!/-8pt/@{>}}
\newdir{>}{{}*:(1,-.2)@^{>}*:(1,+.2)@_{>}}
\newdir{<}{{}*:(1,+.2)@^{<}*:(1,-.2)@_{<}}

\begin{document}

\title{On the normality of Higgins commutators}

\author{Alan S.~Cigoli}
\author{James R.~A.~Gray}
\author{Tim Van~der Linden}

\email{alan.cigoli@unimi.it}
\email{jamesgray@sun.ac.za}
\email{tim.vanderlinden@uclouvain.be}

\thanks{The first author's research was partially supported by FSE, Regione Lombardia. He would like to thank the Institut de Recherche en Math\'ematique et Physique (IRMP) for its kind hospitality during his stay in Louvain-la-Neuve. The third author is a Research Associate of the Fonds de la Recherche Scientifique--FNRS. He would like to thank the University of South Africa for its kind hospitality during his stay in Johannesburg}

\address{Dipartimento di Matematica, Universit\`a degli Studi di Milano, Via Saldini 50, 20133 Milano, Italy}
\address{Mathematics Division, Department of Mathematical Sciences, Stellenbosch University, Private Bag X1, Matieland 7602, South Africa}
\address{Institut de Recherche en Math\'ematique et Physique, Universit\'e catholique de Louvain, chemin du cyclotron~2 bte~L7.01.02, 1348 Louvain-la-Neuve, Belgium}

\begin{abstract}
In a semi-abelian context, we study the condition \NH\ asking that \emph{Higgins commutators of normal subobjects are normal subobjects}. We provide examples of categories that do or do not satisfy this property. We~focus on the relationship with the \emph{Smith is Huq} condition \SH\ and characterise those semi-abelian categories in which both \NH\ and \SH\ hold in terms of reflection and preservation properties of the change of base functors of the fibration of points.
\end{abstract}



\maketitle

\section*{Introduction}
The recent works~\cite{MM-NC, Actions} explain that the universal-algebraic commutator defined by Higgins in the context of varieties of $\Omega$-groups~\cite{Higgins} can be defined in an arbitrary semi-abelian category~\cite{Janelidze-Marki-Tholen}. In contrast with the commutator introduced by Huq~\cite{Huq}---already in a setting essentially equivalent to semi-abelian---which is defined for a pair of subobjects $K$, $L\leq X$ as a \emph{normal} subobject $[K,L]_X \normal X$, a priori the Higgins commutator $[K,L]\leq X$ is in general just a subobject of~$X$, even when $K$ and $L$ are normal subobjects of $X$. In fact the two are closely related, as the former is always the normal closure of the latter.

Since, in general, Huq and Higgins commutators do not coincide, their eventual coincidence, for normal subobjects, becomes a property that a semi-abelian category may or may not satisfy. This condition, which we will denote by \NH, was introduced by the first author in his Ph.D.~thesis~\cite{AlanThesis}. In this article we study the condition \NH\ as well as its relation with other categorical conditions. In addition we give examples, counterexamples and equivalent characterisations.

In Section~\ref{Section Preliminaries} we recall the definitions of the Huq and the Higgins commutator and some of their basic properties. In Section~\ref{Section (NH)} we explain that \NH\ holds if and only if, for any pair of (protosplit) normal subobjects of an object, these two commutators coincide (Theorem~\ref{Theorem Independence of surrounding object Huq}). It is well known that for a group~$G$, the commutator $[G,G]$ is a characteristic subgroup of $G$. In Section~\ref{Section characteristic subobjects} we show that this can be proved in a semi-abelian category satisfying \NH, if the definition of characteristic subobject from~\cite{CigoliMontoliCharSubobjects} is used. In Section~\ref{Categories of interest} we recall from the first author's Ph.D.~thesis~\cite{AlanThesis} that any \emph{category of interest} in the sense of~\cite{Orzech} satisfies \NH\ and we give a first example of a semi-abelian category which does not. In Section~\ref{(SH)} we compare \NH\ with the \emph{Smith is Huq} condition \SH\ considered in~\cite{MFVdL}, and show that the two are independent from each other. In Section~\ref{(NH) + (SH)} we characterise those categories which satisfy both \NH\ and \SH\ in terms of the fibration of points. In particular, Theorem~\ref{Theorem (SH) + (NH)} tells us that \SH~{\rm +}~\NH\ is equivalent to the condition that for any $f\colon{W\to Z}$ in $\C$, the change of base functor $f^{*}\colon \Pt_{Z}(\C)\to \Pt_{W}(\C)$ of the fibration of points preserves Huq commutators of pairs of normal subobjects. In fact, it suffices to have this condition for $W=0$, so \SH\ {\rm +} \NH\ holds if and only if the kernel functors $\Ker\colon \Pt_{Z}(\C)\to \C$ preserve Huq commutators of pairs of normal subobjects.


\section{Preliminaries} \label{Section Preliminaries}
Throughout this paper we assume that $\C$ is a semi-abelian category~\cite{Janelidze-Marki-Tholen, Borceux-Bourn}.
\begin{definition}\cite{Huq}
A pair of morphisms $f\colon A\to C$ and $g\colon B\to C$ is said to \defn{commute} or \defn{cooperate} if
there exists a (necessarily unique) morphism $\varphi$ making the diagram
\[
\xymatrix@!=3ex{
	& A \ar[dl]_-{\langle1_A,0\rangle} \ar [dr]^-{f} \\
	A\times B \ar@{.>}[rr]^-\varphi & & C \\
	& B \ar [ul]^-{\langle0,1_B\rangle} \ar[ur]_-{g}
}\]
commute.
\end{definition}
In this setting it can be seen that two morphisms commute if and only if their regular images commute (see for instance~\cite{Borceux-Bourn}). For this reason we will only define the Huq commutator for subobjects.
\begin{definition}\cite{Huq} \label{def: Huq.com}
For a pair of subobjects $k\colon K\rightarrowtail X$ and $l\colon L\rightarrowtail X$ of $X$ in~$\C$,
the Huq commutator is the smallest normal subobject $[K,L]_X \normal X$ such that the images of $k$ and $l$ commute in the quotient $X/[K,L]_X$.
\end{definition}
\noindent In this context~\cite{Bourn-Huq, Borceux-Bourn} it can be shown that Huq commutator of $k$ and $l$ always exists and can be constructed as the kernel of $q$ in the diagram
\[
\xymatrix@!=7ex{
& K \ar[dl]_-{\langle1_K,0\rangle} \ar @{.>} [d] \ar [dr]^-{k} \\
K\times L \ar @{.>} [r]^-\varphi & Q & X \ar @{.>>} [l]_-q & {[K,L]_X} \ar @{ |>.>} [l]_-{\ker (q)} \\
& L \ar [ul]^-{\langle0,1_L\rangle} \ar @{.>} [u] \ar[ur]_-l
}
\]
where $Q$ is the colimit of the square of solid arrows, or equivalently as the kernel of $q$ in the diagram
$$ \xymatrix@!{
 & K+L \ar [d] \ar @{->>} [r]^-{\bigl\langle\begin{smallmatrix} 1_K & 0\\
0 & 1_L \end{smallmatrix}\bigr\rangle} \ar [d]_-{\mus{k}{l}} & K\times L \ar[d]^{\varphi} \\
[K,L]_X \ar@{ |>->}[r]_-{\ker(q)} &X \ar@{->>} [r]_-{q} & Q \pushout
} $$
in which the right hand square is a pushout.

\begin{definition}\cite{MM-NC, Actions}
For a pair of subobjects $k\colon K\rightarrowtail X$ and $l\colon L\rightarrowtail X$ of an object $X$ in $\C$, the \defn{Higgins commutator} of $K$ and $L$ is the subobject $[K,L]\leq X$ constructed as in diagram
$$ \xymatrix@!=7ex{
K\tensor L \ar @{->>} [d] \ar @{ |>->}[r]^-{\kappa_{K,L}} & K+L \ar [d]^-{\mus{k}{l}} \\
[K,L] \ar @{ >->} [r] 	& X
} $$
where $\kappa_{K,L}$ is the kernel of $\bigl\langle\begin{smallmatrix} 1_K & 0\\
0 & 1_L \end{smallmatrix}\bigr\rangle\colon K+L\to K\times L$.
\end{definition}
The object $K\tensor L$ is the \defn{co-smash product}~\cite{Smash, MM-NC, HVdL} of $K$ and $L$ and the Higgins commutator is its regular image under the morphism $\mus{k}{l}\comp\kappa_{K,L}$. Note that the normal subobject $(K\tensor L) \normal (K+L)$ can also be seen as the Huq commutator of the coproduct inclusions $\iota_K\colon K \to K+L$ and $\iota_L\colon L \to K+L$.

\begin{examples}
In the category $\Gp$ of groups, $K\tensor L$ is generated by words $klk^{-1}l^{-1}$ as a normal subgroup of $K+L$, so that $[K,L]$ is the usual commutator of $K$ and~$L$ in $X$. In the category $\CRng$ of (non-unitary) commutative rings, $[K,L]=KL$.
\end{examples}

We recall:
\begin{lemma}\label{Lemma Independence of surrounding object Higgins}
Consider $K$, $L\leq X\leq Y$. Then the Higgins commutator of $K$ and~$L$, computed in $X$, coincides with the Higgins commutator of $K$ and $L$, computed in~$Y$.
\end{lemma}

\begin{proof}
This follows from uniqueness of regular images.
\end{proof}
\begin{lemma}\label{Lemma Huq construction}
\emph{\cite[Proposition~5.7]{MM-NC}}
For $K$, $L\leq X$, the Huq commutator $[K,L]_X$ is the normal closure in $X$ of the Higgins commutator $[K,L]$.\
\end{lemma}
\begin{proof}
Let $k\colon K\rightarrowtail X$ and $l\colon L\rightarrowtail X$ be subobjects of $X$. Consider the diagram
\[
\xymatrix{
K\tensor L \ar@{ |>->}[r]^-{\kappa_{K,L}}\ar@{->>}[d]_{e} & K+L \ar@{->>}[r]^-{\left\langle\begin{smallmatrix} 1_K & 0\\
0 & 1_L \end{smallmatrix}\right\rangle} \ar[d]^{\mus{k}{l}}& K\times L\ar[d]^{\varphi}\\
[K,L] \ar@{ >->}[r]^-{m} \ar@{ >.>}[d]& X \ar@{->>}[r]_-{q} & Q\pushout\\
[K,L]_X \ar@{ |>->}[ur]_{\ker (q)} &
}
\]
where $m$ is the image of the morphism $\mus{k}{l}$.
Since $\left\langle\begin{smallmatrix} 1_K & 0\\
0 & 1_L \end{smallmatrix}\right\rangle$ is the cokernel of~$\kappa_{K,L}$ and the square on the right is a pushout, it follows that $q$ is the cokernel of $m\comp e$. Since $e$ is an epimorphism, $q$ is also the cokernel of $m$. It immediately follows that the Huq commutator $[K,L]_X$ is the normal closure of $[K,L]$ in~$X$.
\end{proof}

Since in a semi-abelian category the regular image of a normal subobject is normal, if $\mus{k}{l}$ is a regular epimorphism (i.e.~when $X=K\join L$), then the two commutators coincide.

As a consequence of Theorem~\ref{teo: hi.id}, we will see that, for normal subobjects, the two commutators coincide in every \emph{category of interest} in the sense of G.~Orzech~\cite{Orzech}, such as $\Gp$ or the category $\Rng$ of (non-unitary) rings and $R$-$\Lie$ of Lie algebras over a fixed ring $R$. However, Examples \ref{Example Axiom (CI7)} and \ref{Example Axiom (CI8)} show that for an arbitrary semi-abelian category, even those which are closely related to categories of interest, the two commutators need not coincide for abitrary pairs of normal subobjects.

The Higgins commutator can also be used to characterise normal monomorphisms.

\pagebreak
\begin{lemma}\label{Lemma Normal iff commutator factors through}
\emph{\cite[Theorem~6.3]{MM-NC}} A subobject $K\leq X$ is normal in $X$ if and only if $[K,X]\leq K$.\noproof
\end{lemma}

In fact, a more general version of this result holds: Lemma~4.9 in~\cite{Actions}.


\section{The condition \NH} \label{Section (NH)}
In general, as explained above, Huq and Higgins commutators need not coincide; in other words, Higgins commutators need not be normal. In some categories, though, the Huq commutator of a pair of normal monomorphisms will always coincide with its Higgins commutator. In this section we focus on equivalent formulations of this condition.

\begin{definition}\label{Definition (NH)}
A semi-abelian category satisfies the condition \defn{Normality of Higgins commutators (NH)}  when, for each pair of normal subobjects $K$, $L\normal X$, the Higgins commutator $[K,L]\leq X$ is normal in $X$.
\end{definition}

It is easy to check that the category of groups satisfies \NH. By contrast, the following example shows that, for arbitrary (non-normal) subgroups $K$ and $L$ of $X$, the commutator $[K,L]$ need not be a normal subgroup of $X$.

\begin{example}\cite[Example~5.3.9]{AlanThesis}
Let $A_5$ be the simple group of even permutations of order five. For the two subgroups
$K=\langle (12)(34) \rangle$ and $L=\langle (12)(45)\rangle$, we have $[K,L]=\langle (345)\rangle\neq[K,L]_{A_5}=A_5$ its normalisation.
\end{example}

Under~\NH, Higgins commutators and Huq commutators of normal subobjects coincide. 

\begin{lemma}\label{Lemma Higgins = Huq}
In a semi-abelian category with \NH,
\[
K, L\normal X \quad\Rightarrow\quad [K,L]_{X}=[K,L].
\]
\end{lemma}
\begin{proof}
By Lemma~\ref{Lemma Huq construction} the Huq commutator $[K,L]_{X}$ is the normal closure of $[K,L]$ in $X$. Hence if $[K,L]$ is already normal in $X$, then the two commutators will coincide.
\end{proof}

\begin{lemma}\label{Lemma When factors through subobject}
In a semi-abelian category with \NH, if $K$, $L\leq X\leq Y$ and $K$, $L\normal Y$, then
\[
[K,L]_{X}=[K,L]_{Y}.
\]
\end{lemma}
\begin{proof}
If $\C$ satisfies \NH, then it follows from Lemma~\ref{Lemma Higgins = Huq} that both commutators $[K,L]_{X}$ and $[K,L]_{Y}$ coincide with the Higgins commutator $[K,L]$, which is independent of the object in which it is computed by Lemma~\ref{Lemma Independence of surrounding object Higgins}. 
\end{proof}

\begin{lemma}\label{Lemma Reduction to protosplit mono}
Given $K\normal X\normal Y$, consider the induced diagram
\[
\xymatrix{& K \ar@{{ >}->}[rd] \ar@{{ |>}->}[d]_-{k} \\
0 \ar[r] & X \ar@{{ |>}->}[r]_-{\langle x,0\rangle} \ar@{=}[d] & Y\times_{Y/X}Y \splitpullback \ar@<.5ex>[d]^(.6){\pi_{1}} \ar@<.5ex>[r]^-{\pi_{2}} & Y \ar@{-{ >>}}[d] \ar@<.5ex>[l] \ar[r] & 0 \\
0 \ar[r] & X \ar@{{ |>}->}[r]_{x} & Y \ar@{-{ >>}}[r] & Y/X \ar[r] & 0.}
\]
If $K\normal Y$ then $K\normal (Y\times_{Y/X}Y)$.
\end{lemma}
\begin{proof}
Since $K\normal Y$, the right hand side pullback square decomposes into a composite of pullbacks:
\[
\xymatrix{Y\times_{Y/X}Y \pullback \ar@{-{ >>}}[r] \ar@{-{ >>}}[d] & Q \pullback \ar@{-{ >>}}[r] \ar@{-{ >>}}[d] & Y \ar@{-{ >>}}[d]\\
Y \ar@{-{ >>}}[r] & Y/K \ar@{-{ >>}}[r] & Y/X}
\]
It is clear that $K$ is the kernel of ${Y\times_{Y/X}Y \to Q}$.
\end{proof}
Recall the well known fact:
\begin{lemma}\label{Lifting split extensions}
Consider a split extension as in bottom row of the diagram
\[
\xymatrix{0 \ar@{.>}[r] & K \ar@{{ |>}->}[d]_-{k} \ar@{.>}[r] & K' \ar@{.>}[d] \ar@{.>}@<.5ex>[r] & Z \ar@{:}[d] \ar@{.>}@<.5ex>[l] \ar@{.>}[r] & 0\\
0 \ar[r] & X \ar[r]_-{x} & Y \ar@<.5ex>[r]^-{f} & Z \ar@<.5ex>[l]^-{s} \ar[r] & 0}
\]
such that $x\comp k$ is normal. Then this split extension lifts along $k\colon K\to X$ to yield a normal monomorphism of split extensions.
\end{lemma}
\begin{proof}
The needed lifting is obtained via the pullback of split extensions in the diagram
\[
\xymatrix@!0@R=3.5em@C=5em{
& K \skewpullback \ar@{=}[dl] \ar@{{ |>}->}[ddd]|(.33){\hole}_(.6){k} \ar@{{ |>}.>}[rr] && K' \skewpullback \ar[ld] \ar@{.>}[ddd] \ar@{.>}@<0.5ex>[rr] & & Z \skewpullback \ar@{:}[ddd] \ar[dl]_{s} \ar@{.>}@<0.5ex>[ll] \\
K \ar@{{ |>}->}[rr] \ar@{{ |>}->}[ddd]_{x\circ k} & & R \ar@<0.5ex>[rr]^(0.75){r_{1}} \ar[ddd]_(.4){\langle r_{1},r_{2}\rangle } & & Y \ar@<0.5ex>[ll] \ar@{=}[ddd] &\\
& & & & &\\
& X \ar@{{ |>}->}[rr]_(.25){x}|-{\hole} \ar@{{ |>}->}[dl]^(.4){x} & & Y\ar@<0.5ex>[rr]^(.25){f}|-{\hole}\ar[dl]^(.3){\langle s\circ f,1_{Y}\rangle}& & Z \ar@<0.5ex>[ll]^(.75){s}|-{\hole} \ar[dl]^(.4){s}\\
Y \ar@{{ |>}->}[rr]_{\langle 0,1_{Y}\rangle} & & Y\times Y \ar@<0.5ex>[rr]^{\pi_1} & & Y \ar@<0.5ex>[ll]^{\langle 1_{Y},1_{Y}\rangle}}
\]
where $R$ is the denormalisation~\cite{Bourn2000, Borceux-Bourn} of $x\comp k$.
\end{proof}

\begin{definition}
\cite{BJK2} A morphism $\xymatrix@1{K \ar@{ |>->}[r] & X}$ is called a \defn{protosplit (normal) monomorphism} if and only if it is the kernel of a split epimorphism. We will use the notation $K \protosplit X$ to indicate that $K\leq X$ is a \defn{protosplit normal subobject} of $X$, i.e.~its representing monomorphisms are protosplit normal.
\end{definition}

In what follows we shall consider the diagram
\begin{equation}\label{Cospan in points}\tag{$\dagger$}
\vcenter{\xymatrix{0 \ar[r] & K \ar@{{ |>}->}[d]_-{k} \ar[r] & K' \ar[d]^-{k'} \ar@<.5ex>[r] & Z \ar@{=}[d] \ar@<.5ex>[l] \ar[r] & 0\\
0 \ar[r] & X \ar[r]_-{x} & Y \ar@<.5ex>[r]^-{f} & Z \ar@<.5ex>[l]^-{s} \ar[r] & 0\\
0 \ar[r] & L \ar@{{ |>}->}[u]^-{l} \ar[r] & L' \ar[u]_-{l'} \ar@<.5ex>[r] & Z \ar@{=}[u] \ar@<.5ex>[l] \ar[r] & 0}}
\end{equation}
where $k'$ and $l'$ are normal monomorphisms in $\Pt_Z(\C)$.

\begin{theorem}\label{Theorem Independence of surrounding object Huq}
For a semi-abelian category $\C$, the following are equivalent:
\begin{enumerate}
\item $\C$ satisfies \NH;
\item for all $K$, $L\leq X\leq Y$, if $K$, $L\normal Y$, then $[K,L]_{X}=[K,L]_{Y}$;
\item for all $K$, $L\leq X\normal Y$, if $K$, $L\normal Y$, then $[K,L]_{X}=[K,L]_{Y}$;
\item for all $K$, $L\leq X\protosplit Y$, if $K$, $L\normal Y$ then $[K,L]_{X}=[K,L]_{Y}$;
\item for all $K$, $L\leq X\normal Y$, if $K$, $L\normal Y$ and $X=K\join L$, then $[K,L]_{X}=[K,L]_{Y}$;
\item for all $K$, $L\leq X\protosplit Y$, if $K$, $L\normal Y$ and $X=K\join L$, then $[K,L]_{X}=[K,L]_{Y}$;
\item for any diagram \eqref{Cospan in points} there exists a normal subobject $N\normal Y$ in $\Pt_{Z}(\C)$ such that $\Ker(N\normal Y)=[K,L]_X \normal X$;
\item for any diagram \eqref{Cospan in points} such that $X=K\join L$, there exists a normal subobject $N\normal Y$ in $\Pt_{Z}(\C)$ such that $\Ker(N\normal Y)=[K,L]_X \normal X$.
\end{enumerate}
\end{theorem}
\begin{proof}
By Lemma~\ref{Lemma When factors through subobject} (i) implies (ii). It is also clear that (ii) implies (iii), (iii) implies (iv), (iii) implies (v) and (v) implies (vi). Assuming now that (v) holds, since $[K,L]_{K\join L}=[K,L]$ as explained above, (i) follows from the fact that since $K\join L \rightarrowtail Y$ is the join of two normal monomorphisms it is normal~\cite{Borceux-Semiab, Huq}.

By Lemma~\ref{Lemma Reduction to protosplit mono}, condition (vi) implies (v) as follows. Assuming that $K$ and $L$ are normal in $Y$, the lemma gives us $K$, $L\normal (Y\times_{Y/X}Y)$. Hence by the assumption~(vi) and the fact that $\langle x,0\rangle\colon X\to Y\times_{Y/X}Y$ is a protosplit normal monomorphism, we have $[K,L]_{X}=[K,L]_{Y\times_{Y/X}Y}$. Consider the diagram
\[
\xymatrix@=8ex{
& K\tensor L \ar@{ |>->}[rr]^{\kappa_{K,L}}\ar@{->>}[d]_{e} & & K+L \ar[d]\ar@/^8ex/[dd]^{\mus{x\circ k}{x\circ l}}\ar@{->>}[dl]_{\mus{k}{l}}\\
&[K,L] \ar@{ |>->}[r]_-{m}& X \ar@{ |>->}[r]_-{\langle x,0\rangle} \ar[dr]_{x} & Y\times_{Y/X} Y \ar@{->>}[d]^{\pi_1}\\
& & & Y.
}
\]
Since $X=K\join L$ it follows that $[K,L]_X=[K,L]$ and so $[K,L]=[K,L]_{Y \times_{Y/X} Y}$ is normal in $Y \times_{Y/X} Y$. Since the image of a normal monomorphism is normal it follows that $x\comp m$ is normal in $Y$. Therefore since $[K,L]_Y \normal Y$ is the normal closure of $[K,L] \leq Y$ it follows that $[K,L]_Y=[K,L]=[K,L]_X$ as required. Finally we note that (vii) is equivalent to (iv) and (viii) is equivalent to (vi) since they are simple reformulations obtained using Lemma~\ref{Lifting split extensions} and the fact that for a morphism of split extensions
\[
\xymatrix{0 \ar[r] & K \ar@{{ |>}->}[d]_-{k} \ar[r] & K' \ar[d]^{k'} \ar@<.5ex>[r] & Z \ar@{=}[d] \ar@<.5ex>[l] \ar[r] & 0\\
0 \ar[r] & X \ar[r]_-{x} & Y \ar@<.5ex>[r]^-{f} & Z \ar@<.5ex>[l]^-{s} \ar[r] & 0,}
\]
the monomorphism $k'$ considered as a morphism in $\Pt_{Z}(\C)$ is normal if and only if~$x\comp k$ is normal in~$\C$.
\end{proof}


\section{Characteristic subobjects and \NH}\label{Section characteristic subobjects}

In the category of groups it is well known that for each group $G$ the commutator $[G,G]$ is a characteristic subgroup of $G$. It is not difficult to see that a subgroup~$S$ of~$G$ is characteristic if and only if for each group $B$, every action $B\flat G \to G$, defined with respect to the monad $B\flat (-)$ (see~\cite{BJK}, \cite{BJK2} and~\cite{Bourn-Janelidze:Semidirect}), restricts to an action ${B\flat S \to S}$. This description was used by the first author and A.~Montoli in~\cite{CigoliMontoliCharSubobjects} as a definition of characteristic subobject in an arbitrary semi-abelian category. In this section we will give alternative characterisations of characteristic subobjects and then show that \NH\ implies that when $K$ and~$L$ are characteristic subobjects of~$X$ the Huq commutator of $[K,L]_X$ is a characteristic subobject of~$X$. The same result was proved in~\cite{CigoliMontoliCharSubobjects} in a context including categories of interest.
\begin{definition}\cite{CigoliMontoliCharSubobjects}
A subobject $S\leq X$ is said to be \defn{characteristic} when every action ${B\flat X \to X}$ restricts to an action $B\flat S \to S$.
\end{definition}

\begin{proposition}\label{Proposition characterisation of characteristic subobjects}
For a subobject $S \leq X$ the following are equivalent:
\begin{enumerate}
\item $S$ is a characteristic subobject of $X$;
\item each split extension
\[
\xymatrix{X \ar[r]^{x} & Y \ar@<0.5ex>[r]^{f} & Z \ar@<0.5ex>[l]^{s}}
\]
lifts to a morphism of split extensions
\[
\xymatrix{0 \ar@{.>}[r] & S \ar@{{ >}->}[d] \ar@{.>}[r] & T \ar@{.>}[d] \ar@{.>}@<.5ex>[r] & Z \ar@{:}[d] \ar@{.>}@<.5ex>[l] \ar@{.>}[r] & 0\\
0 \ar[r] & X \ar[r]_-{x} & Y \ar@<.5ex>[r]^-{f} & Z \ar@<.5ex>[l]^-{s} \ar[r] & 0}
\]
\item for each $Y$ such that $X\protosplit Y$, $S$ is a normal subobject of $Y$;
\item for each $Y$ such that $X\normal Y$, $S$ is a normal subobject of $Y$.
\end{enumerate}
\end{proposition}
\begin{proof}
The implications (i) $\Leftrightarrow$ (ii) and (i)~$\Rightarrow$~(iv) were proved in~\cite{CigoliMontoliCharSubobjects}. It follows that the proof will be complete if we show that (iv)~$\Rightarrow$~(iii)~$\Rightarrow$~(ii). Trivially (iv)~$\Rightarrow$~(iii) since (iii) is a special case of (iv). Finally, the implication (iii)~$\Rightarrow$~(ii) follows from Lemma~\ref{Lifting split extensions}.\end{proof}

\begin{proposition}\label{Proposition commutator of characteristic subobjects characteristic}
If $\C$ satisfies \NH, then for $K$ and $L$ characteristic subobjects of $X$, $[K,L]_X$ is a characteristic subobject of $X$.
\end{proposition}
\begin{proof}
By Proposition~\ref{Proposition characterisation of characteristic subobjects} it is sufficient to show that if $X$ is a normal subobject of $Y$ then $[K,L]_X$ is a normal subobject of $Y$. But, trivially by Proposition~\ref{Proposition characterisation of characteristic subobjects} we have $K$ and $L$ are normal subobjects of $Y$ and so by Theorem~\ref{Theorem Independence of surrounding object Huq}, $[K,L]_X=[K,L]_Y$ is a normal subobject of $Y$. 
\end{proof}
\begin{corollary}
If $\C$ satisfies \NH, then for each object $X$, the commutator $[X,X]=[X,X]_X$ is a characteristic subobject of $X$.
\end{corollary}
\begin{proof}
This follows from Proposition~\ref{Proposition commutator of characteristic subobjects characteristic} with $K=L=X$, since $X$ is trivially a characteristic subobject of $X$.
\end{proof}

\section{Categories of interest satisfy \NH}\label{Categories of interest}

The condition~\NH\ was first studied in the first author's thesis~\cite{AlanThesis}, where it is also shown that any \emph{category of interest} in the sense of Orzech satisfies it. In this section we recall this fact.

\begin{definition}\cite{Orzech}\label{cat.int}
A \defn{category of interest} is a variety of universal algebras whose theory contains a unique constant $0$, a set $\Omega$ of finitary operations and a set of identities $\mathbb{E}$ such that:
\begin{enumerate}
	\item[\CI1] $\Omega=\Omega_0\cup\Omega_1\cup\Omega_2$, where $\Omega_i$ is the set of $i$-ary operations;
	\item[\CI2] $\Omega_0=\{0\}$, $-\in\Omega_1$ and $+\in\Omega_2$, where $\Omega_i$ is the set of $i$-ary operations, and $\mathbb{E}$ includes the group laws for $0$, $-$, $+$; define $\Omega_1'=\Omega_1\setminus\{-\}$, $\Omega_2'=\Omega_2\setminus\{+\}$;
	\item[\CI3] for any $*\in\Omega_2'$, the set $\Omega_2'$ contains $\bullet$ defined by $x\bullet y=y*x$;
	\item[\CI4] for any $\omega\in\Omega_1'$, $\mathbb{E}$ includes the identity $\omega(x+y)=\omega(x)+\omega(y)$;
	\item[\CI5] for any $*\in\Omega_2'$, $\mathbb{E}$ includes the identity $x*(y+z)=x*y+x*z$;
	\item[\CI6] for any $\omega\in\Omega_1'$ and $*\in\Omega_2'$, $\mathbb{E}$ includes the identity $\omega(x)*y=\omega(x*y)$;
	\item[\CI7] for any $*\in\Omega_2'$, $\mathbb{E}$ includes the identity $x+(y*z)=(y*z)+x$;
	\item[\CI8] for any $*$, $\sqbullet\in\Omega_2'$, there exists a word $w$ such that $\mathbb{E}$ includes the identity
		$$ (x*y)\sqbullet z=w(x*_1(y\sqbullet_1 z),\ldots, x*_m(y\sqbullet_m z),y*_{m+1}(x\sqbullet_{m+1} z),\ldots,y*_n(x\sqbullet_n z)) $$
		where $*_1,\ldots,*_n$ and $\sqbullet_1,\ldots,\sqbullet_n$ are operations in $\Omega_2'$
\end{enumerate}
\end{definition}
Note that we have slightly generalized the original axiom \CI8.
Note also that since any \emph{category of interest} is a variety of $\Omega$-groups, it is automatically semi-abelian~\cite{Janelidze-Marki-Tholen}. 

\begin{examples}\label{examples categories of interest}
The categories $\Gp$ of groups, $\Rng$ of non-unitary rings, and $\K$-$\Lie$ of Lie algebras, $\K$-$\Leibniz$ of Leibniz algebras, and $\K$-$\Poisson$ of Poisson algebras over a field $\K$ are all \emph{categories of interest}. On the other hand, categories of Jordan algebras and of non-associative algebras aren't, because \CI8 fails for them. More generally, varieties of distributive $\Omega$-groups~\cite{Higgins} need not be \emph{categories of interest}.
\end{examples}

\begin{remark} \label{pr.com}
By \CI5, for any $*\in\Omega_2'$, $\mathbb{E}$ includes the identity
\begin{multline*}
(x*z)+(x*t)+(y*z)+(y*t)\\
=(x+y)*(z+t)=(x*z)+(y*z)+(x*t)+(y*t);
\end{multline*}
hence it also includes $(x*t)+(y*z)=(y*z)+(x*t)$.
\end{remark}

\begin{theorem}\emph{\cite[Theorem 5.3.6]{AlanThesis}} \label{teo: hi.id}
If $\C$ is a \emph{category of interest}, $X$ an object in~$\C$ and $K$, $L\normal X$, then the Higgins commutator $[K,L]$ is normal in $X$.
\end{theorem}
\begin{proof}
Since categories of interest are distributive $\Omega$-groups by definition, according to Theorem~4B in~\cite{Higgins}, the commutator $[K,L]$ is the ideal of $K \join L$ generated by the elementary commutator words $\overline{w}(\k,\l)=-w(\k)-w(\l)+w(\k+\l)$, with $\k\in K^n$, $\l\in L^n$ and $w$ a single $n$-ary operation (for any $n$) or the identity. Since $\C$ is a \emph{category of interest}, we have elementary commutator words of these types:
\begin{enumerate}
	\item[(i)] $-k-l+k+l$, with $k\in K$ and $l\in L$;
	\item[(ii)] $-\omega(k)-\omega(l)+\omega(k+l)$, with $k\in K$, $l\in L$ and $\omega\in\Omega_1'$;
	\item[(iii)] $-(k_1*k_2)-(l_1*l_2)+((k_1+l_1)*(k_2+l_2))$, with $k_1$, $k_2\in K$, $l_1$, $l_2\in L$ and $*\in\Omega_2'$.
\end{enumerate}
Now, by Definition~\ref{cat.int},
$$ -\omega(k)-\omega(l)+\omega(k+l)\stackrel{\CI4}{=}-\omega(k)-\omega(l)+\omega(k)+\omega(l) $$
and $\omega(k)\in K$, $\omega(l)\in L$, since $K$ and $L$ are subobjects. It follows that words of type (ii) are again of type (i). Moreover,
\begin{multline*}
	-(k_1*k_2)-(l_1*l_2)+((k_1+l_1)*(k_2+l_2)) \\
	\stackrel{\CI5}{=}-(k_1*k_2)-(l_1*l_2)+(k_1*k_2)+(k_1*l_2)+(l_1*k_2)+(l_1*l_2) \\
	=(k_1*l_2)+(l_1*k_2),
\end{multline*}
where the last equality follows from Remark~\ref{pr.com}.

So we only have two types of elementary commutator words:
\begin{enumerate}
	\item[(i)] $-k-l+k+l$, with $k\in K$ and $l\in L$;
	\item[(iii)'] $(k_1*l_2)+(l_1*k_2)$, with $k_1$, $k_2\in K$, $l_1$, $l_2\in L$ and $*\in\Omega_2'$.
\end{enumerate}
Now let $\CKL$ be the subobject of $X$ generated by the elementary commutator words. We shall prove that $\CKL$ is normal in $X$ (and thus in $K \join L$). As a consequence, $\CKL=[K,L]$ and $[K,L]$ will be an ideal of $X$, as required.

In order to prove that $\CKL$ is an ideal of $X$, by Theorem 1.7 in~\cite{Orzech} it suffices to prove that it is closed under conjugation and products with elements of $X$. We start by verifying these two conditions for generators.

(i) For any $k\in K$, $l\in L$, $x\in X$:
		\begin{multline*}
		-x+(-k-l+k+l)+x\\
		=(-x-k+x)+(-x-l+x)+(-x+k+x)+(-x+l+x)
		\end{multline*}
		and the right hand expression is an elementary commutator word of type (i), $K$ and $L$ being ideals of $X$;
	for any $k\in K$, $l\in L$, $x\in X$ and $*\in\Omega_2'$:
		$$ (-k-l+k+l)*x\stackrel{\CI5}{=}-(k*x)-(l*x)+(k*x)+(l*x) $$
		and again the right hand expression is an elementary commutator word of type~(i), $K$ and $L$ being ideals of $X$.

(iii)' For any $k_1$, $k_2\in K$, $l_1$, $l_2\in L$, $x\in X$ and $*\in\Omega_2'$
		$$ -x+(k_1*l_2)+(l_1*k_2)+x\stackrel{\CI7}{=}(k_1*l_2)+(l_1*k_2), $$
		that is, commutator words of type (iii)' are stable by conjugation;
	for any $k_1$, $k_2\in K$, $l_1$, $l_2\in L$, $x\in X$ and $*$, $\sqbullet\in\Omega_2'$,
		$$ ((k_1*l_2)+(l_1*k_2))\sqbullet x\stackrel{\CI5}{=}((k_1*l_2)\sqbullet x)+((l_1*k_2)\sqbullet x) $$
		but, by \CI8, there exists a word $w$ such that
\begin{multline*}
(k_1*l_2)\sqbullet x=\\
w(k_1*_1(l_2\sqbullet_1 x),\ldots, k_1*_m(l_2\sqbullet_m x),l_2*_{m+1}(k_1\sqbullet_{m+1} x),\ldots,l_2*_n(k_1\sqbullet_n x))
\end{multline*}
		where each term on the right is an elementary commutator word since $K$ and $L$ are ideals, and so the product
		$(k_1*l_2)\sqbullet x$ is generated by elementary commutator words. Similarly, the same holds for the product ${(l_1*k_2)\sqbullet x}$.

We conclude the proof by induction. Let $w_1(\k,\l)$ and $w_2(\k',\l')$ be words in $\CKL$ satisfying the two conditions above---conjugates and products with elements of $X$ are still in $\CKL$. Let us first consider the sum $w_1+w_2$. For any $x\in X$,
		$$ -x+(w_1+w_2)+x=(-x+w_1+x)+(-x+w_2+x) $$
		and the right hand expression is in $\CKL$, since $w_1$ and $w_2$ satisfy the induction hypothesis. For any $x\in X$ and $*\in\Omega_2'$,
		$$ (w_1+w_2)*x\stackrel{\CI5}{=}(w_1*x)+(w_2*x) $$
and again the right hand expression is in $\CKL$, since $w_1$ and $w_2$ satisfy the induction hypothesis.
		
Now consider the product $w_1*w_2$ where $*$ is a fixed operation in $\Omega_2'$. For any $x\in X$,
		$$ -x+(w_1*w_2)+x\stackrel{\CI7}{=}w_1*w_2, $$
		that is, the product $w_1*w_2$ is stable under conjugation;
and for any $x\in X$ and $\sqbullet\in\Omega_2'$,
		\begin{align*} (w_1*w_2)\sqbullet x
 \stackrel{\CI8}{=}w(&w_1*_1(w_2\sqbullet_1 x),\ldots, w_1*_m(w_2\sqbullet_m x),\\
&w_2*_{m+1}(w_1\sqbullet_{m+1} x),\ldots,w_2*_n(w_1\sqbullet_n x)) \end{align*}
where each term on the right is generated by elementary commutator words, because $w_1$ and $w_2$ satisfy the induction hypothesis. Hence $(w_1*w_2)\sqbullet x\in \CKL$, which concludes the proof.
\end{proof}

Next we give an example showing that~\NH\ need not hold outside the context of \emph{categories of interest}. In particular \CI7 is necessary. Example~\ref{Example Axiom (CI8)} below proves that also \CI8 is also necessary. Both show that there are semi-abelian categories in which Higgins commutators of normal subobjects need not be normal.

\begin{example}\label{Example Axiom (CI7)}~\cite[Example 5.3.8]{AlanThesis}
Consider the category whose objects are groups with an additional binary associative and distributive operation $\ast$, which satisfies all the axioms in Definition~\ref{cat.int}, except for \CI7, and morphisms are group homomorphisms preserving~$\ast$. Let $A$ be the dihedral group of order 8, in additive notation:
$$ A=\langle r,s \mid 4r=0,\ 2s=0,\ s+r+s=3r\rangle $$
endowed with an associative and distributive product generated by:
$$
\begin{array}{r|cc}
	\ast & r & s \\
	\hline
	r & s & s \\
	s & s & s \\
\end{array}
$$
The subobject $K=\{0,2r,s,s+2r\}$ of $A$ generated by $s$ and $2r$ is an ideal of~$A$, whereas the commutator $[K,K]=\{0,s\}$ is not, since it is not closed under conjugation: indeed, $r+s-r=r+s+3r=s+2r\not \in [K,K]$.
\end{example}


\section{Independence of the \emph{Smith is Huq} condition \SH} \label{(SH)}

In this section we prove that \NH\ is independent of the \emph{Smith is Huq} condition~\SH, by giving examples of categories which satisfy one but not the other.

\subsection{The \emph{Smith is Huq} condition \SH}\label{sub (SH)}
Given two equivalence relations $R$ and~$S$ on $X$, with respective normalisations $K$, $L\normal X$, the Smith commutator $[R,S]^{\S}$ of $R$ and~$S$ is an equivalence relation on $X$ which measures how far $R$ and~$S$ are from centralising each other (see~\cite{Smith, Pedicchio, Borceux-Bourn}). If the Smith commutator of two equivalence relations is trivial, then the Huq commutator of their normalisations is also trivial~\cite{BG}. But, in general, the converse is false; in~\cite{Borceux-Bourn, Bourn2004} a counterexample is given in the category of digroups, which is a semi-abelian variety, even a variety of~$\Omega$-groups~\cite{Higgins}. The requirement that the two commutators vanish together is known as the \defn{Smith is Huq Condition (SH)} and it is shown in~\cite{MFVdL} that, for a semi-abelian category, this condition holds if and only if every star-multiplicative graph is an internal groupoid, which is important in the study of internal crossed modules~\cite{Janelidze}. Moreover, \SH\ is also known to hold for pointed strongly protomodular categories~\cite{BG} (in particular, for any Moore category~\cite{Gerstenhaber, Rodelo:Moore}) and in action accessible categories~\cite{BJ07} (in particular, for any \emph{category of interest}~\cite{Montoli}).

\subsection{Characterisation in terms of Higgins commutators}
Given subobjects $k\colon K\to X$, $l\colon L\to X$ and $m\colon M \to X$ of an object $X$, the \defn{ternary Higgins commutator} $[K,L,M]\leq X$ is the image of the composite
\[
\xymatrix@C=3em{K\tensor L\tensor M \ar@{{ |>}->}[r]^-{\iota_{K,L,M}} \ar@{.{ >>}}[d] & K+L+M \ar[d]^-{\muss{k}{l}{m}}\\ [K,L,M] \ar@{{ >}.>}[r] & X} 
\]
where $\iota_{K,L,M}$ is the kernel of
\[
\xymatrix@=6em{K+L+M \ar[r]^-{\left\langle\begin{smallmatrix} i_{K} & i_{K} & 0 \\
i_{L} & 0 & i_{L}\\
0 & i_{M} & i_{M}\end{smallmatrix}\right\rangle} & (K+L)\times (K+M) \times (L+M);}
\]
$i_k, i_L$ and $i_M$ denote coproduct injection morphisms. The object $K\tensor L\tensor M$ is the \defn{ternary co-smash product}~\cite{Smash, HVdL, Actions} of $K$,~$L$ and~$M$. 

The main result of~\cite{HVdL} states that for all $K$, $L\normal X$, the Smith commutator $[K,L]^{\S}$ may be decomposed as the join $[K,L]\join [K,L,X]$, so that \SH\ holds if and only if $[K,L]_{X}=[K,L]\join [K,L,X]$ or, equivalently, $[K,L,X]\leq [K,L]_{X}$.

\subsection{Relation between \NH\ and \SH}

It is a natural question to ask whether the conditions \NH\ and \SH\ are related. The following examples show that they are, in fact, independent.

\begin{example}\label{Example Axiom (CI8)}~\cite[Example 5.3.7]{AlanThesis}
Let $\NARng$ be the category \defn{non-associa\-tive rings}~\cite{Higgins} whose objects are abelian groups with an additional binary operation~$\ast$ which distributes over addition; and whose morphisms are group homomorphisms preserving~$\ast$. This category satisfies all axioms in Definition~\ref{cat.int}, except for \CI8.

Let $A$ be the object in $\NARng$ with abelian group structure the free abelian group on $\{x,y,z\}$, endowed with a distributive product with the following multiplication table:
$$
\begin{array}{r|ccc}
	\ast & x & y & z \\
	\hline
	x & x & 0 & y \\
	y & 0 & 0 & x \\
	z & y & x & z \\
\end{array}
$$
The subobject $K$ generated by $x$ and $y$ is an ideal of $A$, whereas the commutator $[K,K]$, which is the subobject generated by $x$ is not, because it is not closed under multiplication with external elements: $x*z=y\not \in [K,K]$.

Since the category $\NARng$ is strongly protomodular~(being a variety of distributive $\Omega_2$-groups~\cite{MM}), it follows that strong protomodularity~\cite{Bourn-SP, Bourn2004} does not imply \NH. In particular, since a strongly protomodular semi-abelian category always satisfies \SH\, it follows that a semi-abelian category may satisfy \SH\ but not~\NH.
\end{example}

\begin{example}
Let $\C$ be the category whose objects are abelian groups endowed with a symmetric and distributive ternary operation $t$ satisfying the following associativity property:
$$ t(t(x,y,z),u,v)=t(t(x,u,v),y,z). $$
Morphisms are as usual maps preserving all operations.
Since $\C$ is a variety of distributive $\Omega$-groups, we know from~\cite[Theorems 4A, 4C]{Higgins} that, given $K$, $L\leq X$ in $\C$:
\begin{enumerate}
	\item $K \normal X$ if and only if for all $k\in K$ and $x_1,x_2\in X$ also $t(k,x_1,x_2)\in K$;
	\item $[K,L]$ is generated by elements of type $t(k_1,k_2,l_2)$ or $t(k_1,l_1,l_2)$, where $k_1$, $k_2\in K$ and $l_1$, $l_2\in L$.
\end{enumerate}
For $K \leq X$, if $k_1$, $k_2$, $k_3$ are elements of $K$ and $x_1$, $x_2$ are elements of $X$ such that $t(k_i,x_1,x_2) \in K$, then 
\begin{align*}
t(k_1+k_2,x_1,x_2)=t(k_1,x_1,x_2)+t(k_2,x_1,x_2) \in K\\
t(t(k_1,k_2,k_2),x_1,x_2) = t(t(k_1,x_1,x_2),k_2,k_3) \in K.
\end{align*}
Hence it is sufficient to check (i) on generators.
As a consequence, if $K$, $L\normal X$ then $[K,L]$ is normal in $X$, since 
\[
t(t(k_1,k_2,l_2),u,v)=t(t(k_1,u,v),k_2,l_2)\in[K,L]
\]
for all $u$, $v\in X$, and a similar argument holds for the terms of second type. This shows that \NH\ holds in $\C$. 

Consider now the object of $\C$ consisting of the abelian group $\mathbb{Z}$ with the operation $t(x,y,z)=xyz$. Then, if we consider the subobjects $2\mathbb{Z}$ and $4\mathbb{Z}$, it happens that $[2\mathbb{Z},4\mathbb{Z}]=16\mathbb{Z}$, while $[2\mathbb{Z},4\mathbb{Z},\mathbb{Z}]=8\mathbb{Z}$. So $[2\mathbb{Z},4\mathbb{Z}]^{\S}=[2\mathbb{Z},4\mathbb{Z}]\join[2\mathbb{Z},4\mathbb{Z},\mathbb{Z}]>[2\mathbb{Z},4\mathbb{Z}]$ showing that $\C$ does not satisfy \SH.
\end{example}


\section{Equivalent characterisations of \SH\ + \NH} \label{(NH) + (SH)}

Many categories---all \emph{categories of interest}, for instance, as explained in~\ref{teo: hi.id} and~\ref{sub (SH)}---do actually satisfy \emph{both} \SH\ and \NH. These two conditions, when required together, may be characterised in terms of the fibration of points as shown in Theorem~\ref{Theorem (SH) + (NH)}. We begin with a straightforward characterisation in terms of ternary commutators.

\begin{proposition}\label{Ternary characterisation}
A semi-abelian category $\C$ satisfies \SH\ {\rm +} \NH\ if and only if for all $K$, $L\normal X$ in $\C$,
\[
[K,L,X]\leq [K,L].
\]
\end{proposition}
\begin{proof}
This follows from the chain of inclusions
\[
[K,L]\leq [K,L]_{X}\leq [K,L]\join [K,L,X]
\]
and the fact that the Smith commutator of the equivalence relations corresponding to $K$ and $L$ has the join on the right as its normalisation~\cite{HVdL}.
\end{proof}

This immediately implies that any semi-abelian category $\C$ which satisfies\\\SH~{\rm +}~\NH\ is \emph{peri-abelian} in the sense of~\cite{Bourn-Peri}, since, via the characterisation in~\cite{GrayVdL1}, $\C$ is such if and only if for all $K\normal X$ we have $[K,K,X]\leq [K,K]$.

It was proved in~\cite{BMFVdL, MFVdL3} that \SH\ is equivalent to the condition that kernel functors reflect Huq commutativity of normal subobjects. By (vii) in Theorem~\ref{Theorem Independence of surrounding object Huq}, condition \NH\ is equivalent to the condition that Huq commutators of cospans of normal monomorphisms which are the image of cospans of normal monomorphisms under a kernel functor are themselves images of normals subobjects under the same kernel functor. Hence we are able to study these properties together using an abstract functor as in Lemmas~\ref{lemma: reflection of commutativity the same as reflection of commutator} and~\ref{lemma: NH + SH for abtract functor} below, and make conclusions about the condition \NH\ + \SH\ in Theorem~\ref{Theorem (SH) + (NH)}.

\begin{definition}
A class $\CC$ of cospans in $\C$ is \defn{closed under (direct) images} when for any cospan $(k\colon K\to X,l\colon L\to X)$ in $\CC$ and any regular epimorphism $e\colon X\to X'$ in $\C$, the cospan $(k',l')$ where $k'$ and $l'$ are the images of $e\comp k$ and $e\comp l$, respectively, is in $\CC$.
\end{definition}

Recall that a functor is said to be \defn{conservative} when it reflects isomorphisms.

\begin{lemma}\label{lemma: reflection of commutativity the same as reflection of commutator}
Let $\C$ and $\D$ be semi-abelian categories, let $\CC$ be a class of cospans in $\C$ which is closed under images, let $F\colon \C\to \D$ be a conservative functor which preserves limits and regular epimorphisms, and let $\DD$ be the image of $\CC$ under $F$. The following are equivalent: 
\begin{enumerate}
\item $F$ reflects Huq commutativity of those cospans in $\DD$;
\item $F$ reflects Huq commutators of those cospans in $\DD$.
\end{enumerate} 
\end{lemma}
\begin{proof}
The implication (ii)~$\Rightarrow$~(i) follows from the fact that the functor $F$ preserves the zero object. To prove (i)~$\Rightarrow$~(ii), let $(k\colon K\to X,l\colon L\to X)$ be a cospan in $\CC$ and suppose there exists a normal monomorphism $w\colon W\to X$ such that $F(w)$ is the Huq commutator of $F(k)$ and $F(l)$. Let $e$ be the cokernel of $w$ as displayed in the the short exact sequence
\[
\xymatrix{
0\ar[r] & W \ar[r]^-{w} & X \ar[r]^-{e} & X/W \ar[r] & 0
}
\]
and let $k'$ and $l'$ be the images of $e\comp k$ and $e\comp l$.
Since $F$ preserves limits and regular epimorphisms it preserves short exact sequences (since regular epimorphisms are normal in $\D$), and so $F(e)$ is the quotient of the Huq commutator of $F(k)$ and $F(l)$, which by definition means that $F(e\comp k)$ and $F(e\comp l)$ and so by~\cite{Borceux-Bourn} that $F(k')$ and $F(l')$ commute. Since $\CC$ is closed under images, $(k',l')$ is in $\CC$ and so, by (i), $k'$ and $l'$ and therefore $e\comp k$ and $e\comp l$ commute. It follows that $[K,L]_X \leq W$ and therefore that $F([K,L]_X) \leq F(W)$. Since $F$ preserves Huq commutativity (since it preserves limits) and short exact sequences it follows that
\[
F(W)=[F(K),F(L)]_{F(X)} \leq F([K,L]_X),
\]
meaning that $F([K,L]_X)=F(W)$, and therefore since $F$ reflects isomorphisms $[K,L]_X=W$ as required.
\end{proof}
\begin{lemma}\label{lemma: NH + SH for abtract functor}
Let $\C$ and $\D$ be semi-abelian categories, let $\CC$ be a class of cospans in $\C$ which is closed under images, let $F\colon \C\to \D$ be a conservative functor which preserves limits and regular epimorphisms, and let $\DD$ be the image of $\CC$ under $F$. 
The following are equivalent: 
\begin{enumerate}
\item $F$ reflects Huq commutativity of cospans in $\DD$, and Huq commutators of cospans in $\DD$ are the image of normal subobjects under $F$;
\item $F$ reflects Huq commutators of cospans in $\DD$, and Huq commutators of cospans in $\DD$ are the image of normal subobjects under $F$;
\item $F$ preserves Huq commutators of cospans in $\CC$.
\end{enumerate} 
\end{lemma}
\begin{proof}
The equivalence of (i) and (ii) follows from Lemma~\ref{lemma: reflection of commutativity the same as reflection of commutator}. Since $F$ reflects isomorphisms, it easily follows that (iii)~$\Rightarrow$~(ii). To prove that (ii)~$\Rightarrow$~(iii), let $(k\colon K\to X, l\colon L\to X)$ be a cospan in $\CC$. By the second part of (ii), there exists a normal monomorphism $w\colon  W\to X$ such that $F(w)$ is the Huq commutator of $F(k)$ and $F(l)$. It follows that since $F$ reflects such commutators, $w$ is the commutator of $k$ and $l$ meaning that $F$ preserves commutators of cospans in $\CC$ as required.
\end{proof}

Now we apply this to the situation where $F=\Ker\colon {\Pt_{Z}(\C)\to \C}$ is a kernel functor and $\CC$ is the class of cospans of normal monomorphisms.

\begin{theorem} \label{Theorem (SH) + (NH)}
For a semi-abelian category $\C$, the following are equivalent:
\begin{enumerate}
\item $\C$ satisfies \SH\ {\rm +} \NH;
\item $\C$ satisfies \NH\ and the kernel functors $\Ker\colon \Pt_{Z}(\C)\to \C$ reflect Huq commutators of pairs of normal subobjects;
\item the kernel functors $\Ker\colon \Pt_{Z}(\C)\to \C$ preserve Huq commutators of pairs of normal subobjects;
\item $\C$ satisfies \NH\ and the change of base functors $f^{*}\colon \Pt_{Z}(\C)\to \Pt_{W}(\C)$ of the fibration of points reflect Huq com\-mu\-ta\-tors of pairs of normal subobjects;
\item the change of base functors $f^{*}\colon \Pt_{Z}(\C)\to \Pt_{W}(\C)$ of the fibration of points preserve Huq com\-mu\-ta\-tors of pairs of normal subobjects;
\item for each $Z$ in $\C$ the category $\Pt_{Z}(\C)$ satisfies \SH\ {\rm +} \NH.
\end{enumerate}
\end{theorem}
\begin{proof}
As explained above, (i) is equivalent to (ii). The equivalence between (ii) and (iii) follows from Lemma~\ref{lemma: NH + SH for abtract functor}. Next we will show that (ii) + (iii) implies (iv)~+~(v). Let $f\colon {W\to Z}$ be a morphism $\C$. Consider the diagram of induced pullback functors
\[
\xymatrix@=4em{\Pt_{Z}(\C) \ar[r]^-{f^{*}} \ar@/_4ex/[rr]_-{!_{Z}^{*}=\Ker} & \Pt_{W}(\C) \ar[r]^-{!_{W}^{*}=\Ker} & \C}
\]
which commutes (up to natural isomorphism).
It is clear that $f^{*}$ preserves Huq commutators of pairs of normal subobjects because the kernel functor $!_{Z}^{*}$ preserves them and $!_{W}^{*}$ reflects them. On the other hand, $f^{*}$ reflects Huq commutators of pairs of normal subobjects because $!_{W}^{*}$ preserves them and $!_{Z}^{*}$ reflects them. The implications (vi)\implies (i), (v)\implies (iii) and (iv)\implies (ii) are obvious. 
Finally, since there is an isomorphism of categories $\Pt_{(A,p,s)}(\Pt_B(\C)) \cong \Pt_{A}(\C)$ making the diagram
\[
\xymatrix{
\Pt_{(A,p,s)}(\Pt_{B}(\C)) \ar[rr]^-{\cong} \ar[dr]_{\Ker} & & \Pt_A(\C)\ar[dl]^{s^*}\\
& \Pt_B(\C) &
}
\]
commute, it follows that (v) for $\C$ implies (iii) for $\Pt_{B}(\C)$ which then implies (i) for $\Pt_B(\C)$.
\end{proof}


\providecommand{\noopsort}[1]{}
\providecommand{\bysame}{\leavevmode\hbox to3em{\hrulefill}\thinspace}
\providecommand{\MR}{\relax\ifhmode\unskip\space\fi MR }
\providecommand{\MRhref}[2]{%
  \href{http://www.ams.org/mathscinet-getitem?mr=#1}{#2}
}
\providecommand{\href}[2]{#2}

\end{document}